\documentclass{amsart}
\newtheorem{theorem}{Theorem}[section]
\newtheorem{lemma}[theorem]{Lemma}
\newtheorem{corollary}[theorem]{Corollary}

\newtheorem{example}[theorem]{Example}
\newtheorem{remark}[theorem]{Remark}
\def\erre{{\rm I\!R}}

\def\RR{{\rm I\!R}}

\def\phi{\varphi}

\usepackage{hyperref}
\title[Nonlinear elliptic equations on Carnot groups]{Nonlinear elliptic equations on Carnot groups}
\author{Massimiliano Ferrara}
\address[Massimiliano Ferrara]{University of Reggio Calabria and CRIOS University Bocconi of Milan, Via dei Bianchi presso Palazzo Zani, 89127 Reggio Calabria, Italy
}
\email{massimiliano.ferrara@unirc.it}
\author{Giovanni Molica Bisci}
\address[Giovanni Molica Bisci]{Dipartimento P.A.U., Universit\`a  degli
Studi Mediterranea di Reggio Calabria, Salita Melissari - Feo di
Vito, 89100 Reggio Calabria, Italy} \email{gmolica@unirc.it}
\author{Du\v{s}an Repov\v{s}}
\address[Du\v{s}an Repov\v{s}]{Faculty of Education, and Faculty of Mathematics and Physics, University of Ljubljana, 1000 Ljubljana, Slovenia
}
\email{dusan.repovs@guest.arnes.si}
\thanks{Typeset by \LaTeX}
\begin{document}

\begin{abstract}
This article concerns a class of elliptic equations on Carnot groups depending on one real positive parameter and involving a subcritical nonlinearity (for the critical case we refer to 	G. Molica Bisci and D. Repov\v{s}, 
{\sl Yamabe-type equations on Carnot groups}, 
Potential Anal. 46:2 (2017), 369-383; \href{https://arxiv.org/abs/1705.10100}{arXiv:1705.10100 [math.AP]}
As a special case of our results we prove the existence of at least one nontrivial solution for a subelliptic equation defined on
a smooth and bounded domain $D$ of the Heisenberg group
$\mathbb{H}^n=\mathbb{C}^n\times \erre$.
The main approach is based on variational methods.
\end{abstract}
\maketitle

\section{Introduction}

Analysis on Carnot-Carath\'{e}odory (briefly CC) spaces is a field currently undergoing great development. These abstract structures are a special class of metric spaces in which
the interactions between analytical and geometric tools have been carried out with prosperous results.\par
 In this setting, a fundamental
role is played by Carnot groups that, as it is well-known, are finite dimensional, simply connected Lie
groups $\mathbb{G}$ whose Lie algebra $\mathfrak{g}$ of left invariant vector fields is stratified (see Section \ref{section2}).
Roughly speaking Carnot groups can be seen as local models of CC spaces. Indeed, they are the natural
{tangent} spaces to CC spaces, exactly as Euclidean spaces are tangent to manifolds.\par

It is well-known that a great attention has been focused by many authors on the study of subelliptic equations on Carnot groups and in particular, on the Heisenberg group $\mathbb{H}^n$. See, among others, the papers \cite{BK, BFP, DM, GaLa,Lo1,PiVa}, as well as \cite{Mingio0, Mingio} and references therein.\par
 Motivated by this large interest, we study here the existence of weak solutions for the following problem
 $$
(P_{\lambda}^{f})\,\,\,\,\,\,\,\,\,\,\left\{
\begin{array}{ll}
-\Delta_{\mathbb{G}} u=\displaystyle \lambda f(\xi,u) &  \mbox{\rm in } D\\
u|_{\partial D}=0, &
\end{array}\right.
$$
 where $D$ is a smooth bounded domain of the Carnot group $\mathbb{G}$, $\Delta_{\mathbb{G}}$ is the subelliptic Laplacian on $\mathbb{G}$, and $\lambda$ is a positive real parameter.\par
  Problem $(P_{\lambda}^{f})$ has a variational nature, hence its weak solutions can be found as critical points of a suitable functional $\mathcal{J}_{\lambda}$ defined on the Folland-Stein space $S^1_0(D)$, whose analytic construction is recalled in Section \ref{section2}.\par

 Thanks to this fact, the main approach is based on the direct methods of calculus of variations (see \cite{MRS}) and on the geometric abstract framework on Carnot groups (see, among others, the classical reference \cite{BLU} and references therein).\par
 More precisely, under a suitable subcritical growth condition on the nonlinear term $f$, we are able to prove the existence of at least one (non-trivial) weak solution of problem $(P_{\lambda}^{f})$ provided that $\lambda$ belongs to a precise bounded interval of positive parameters.\par
  The main novelty of this new framework is
that, instead of the usual assumptions on functionals, it requires some hypotheses on the
nonlinearity, which allow for better understanding of the existence phenomena. This allows us to enlarge
the set of applications of the direct minimization exploiting this abstract method
without any asymptotic condition of the term $f$ at zero, as requested in \cite[Theorem 3.1]{GaLa}.\par
  A special case of our result, in the Heisenberg setting, reads as follows.
\begin{theorem}\label{FerraraMolicaBisci22Generale2}
Let $D$ be a smooth and bounded domain of the {Heisenberg group}
$\mathbb{H}^n$ and let
$f:\erre\rightarrow\erre$ be a continuous function such that
\begin{equation}\label{growth}
\sup_{t\in\erre}\frac{|f(t)|}{1+|t|^{p}}\leq \kappa,
\end{equation}
where $p\in (1,\gamma 2^*_h-1)$, with $\gamma\in (2/2^*_h,1)$ and $2^*_h:=\displaystyle 2\left(\frac{n+1}{n}\right)$.
Assume that
\begin{equation}\label{growth2}
0<\kappa<\frac{(p-1)^{\frac{p-1}{p}}}{pc_{1,\gamma}^{\frac{p-1}{p}}c_{2,\gamma}^{\frac{p+1}{p}}}|D|^{\frac{1-\gamma}{p}+
\frac{(1-p)(\gamma 2^{*}_h-1)}{p\gamma 2^{*}_h}
},
\end{equation}
 where $c_{1,\gamma}$ and $c_{2,\gamma}$ denote the embedding constants of the Folland-Stein space $\mathbb{H}^1_0(D)$ in $L^{{\gamma 2^{*}_h}}(D)$ and $L^{\frac{p+1}{\gamma}}(D)$, respectively.\par
\indent Then
the following subelliptic problem
$$
(P_{\kappa})\,\,\,\,\,\,\,\,\,\,\left\{
\begin{array}{ll}
-\Delta_{\mathbb{H}^n} u=\displaystyle f(u) &  \mbox{\rm in } D\\
u|_{\partial D}=0, &
\end{array}\right.
$$
has a weak solution $u_{0,\kappa}\in \mathbb{H}^1_0(D)$ such that $$\|u_{0,\kappa}\|_{\mathbb{H}^1_0(D)}< \left(\kappa p\kappa_{2,\gamma}^{p+1}|D|^{1-\gamma}\right)^{\frac{1}{1-p}}.$$
\end{theorem}

For the sake of completeness we recall that very recently, in \cite{MR}, the existence of multiple solutions for parametric elliptic equations on Carnot groups has been
proved by exploiting the celebrated Ambrosetti-Rabinowitz condition and a local minimum result due to Ricceri (see \cite{R2}). We emphasize that in the present paper we do not require such technical assumption for the nonlinear term $f$. Moreover, the results obtained here are completely different from the one contained in \cite{MRe}
(see also \href{https://arxiv.org/abs/1705.10100}{arXiv:1705.10100 [math.AP]}), where critical subelliptic problems on Carnot groups was studied.\par

\indent The plan of the paper is as follows. Section \ref{section2} is devoted to our abstract framework and preliminaries. Next, in Section \ref{Section3}, Theorem \ref{MoReGeneral} and some preparatory results (see Lemmas \ref{lemmino1} and \ref{lemmino2}) are presented. In the last section, Theorem \ref{MoReGeneral} is proved.\par

\section{Abstract Framework}\label{section2}

In this section we briefly recall some basic facts on Carnot groups and the functional space~$S^1_0(D)$.\par
\smallskip
\textbf{Dilatations}. Let $(\erre^n,\circ)$ be a Lie group equipped with a family of group automorphisms, namely \textit{dilatations}, $\mathfrak{F}:=\{\delta_\eta\}_{\eta>0}$ such that, for every $\eta>0$, the map
$$
\delta_\eta:\prod_{k=1}^{r}\erre^{n_k}\rightarrow\prod_{k=1}^{r}\erre^{n_k}
$$
is given by
$$
\delta_\eta(\xi^{(1)},...,\xi^{(r)}):=(\eta \xi^{(1)}, \eta^{2}\xi^{(2)},...,\eta^r\xi^{(r)}),
$$
where $\xi^{(k)}\in \erre^{n_k}$ for every $k\in \{1,...,r\}$ and $\displaystyle \sum_{k=1}^{r}n_k=n$.\par
\smallskip
\textbf{Homogeneous dimension}. The structure $\mathbb{G}:=(\erre^n,\circ, \mathfrak{F})$ is called a \textit{homogeneous} group with \textit{homogeneous dimension}
\begin{equation}\label{dim}
\textrm{dim}_h{\mathbb{G}}:=\displaystyle \sum_{k=1}^{r}kn_k.
\end{equation}
\noindent From now on, we shall assume that $\textrm{dim}_h{\mathbb{G}}\geq 3$. We remark that, if $\textrm{dim}_h{\mathbb{G}}\leq 3$, then
necessarily $\mathbb{G}=(\erre^{\textrm{dim}_h{\mathbb{G}}},+)$.
Note that the number $\textrm{dim}_h{\mathbb{G}}$ is naturally associated to the family $\mathfrak{F}$ since, for every $\eta>0$, the Jacobian of the map
$$
\xi\mapsto \delta_\eta(\xi),\quad \forall\,\xi\in \erre^n
$$
equals $\eta^{\textrm{dim}_h{\mathbb{G}}}$.\par
\smallskip
\textbf{Stratification}. Let $\mathfrak{g}$ be the Lie algebra of left invariant vector fields on $\mathbb{G}$ and assume that $\mathfrak{g}$ is \textit{stratified}, that is:
$$
\displaystyle\mathfrak{g}=\bigoplus_{k=1}^{r}V_k,
$$
where the integer $r$ is called the \textit{step} of $\mathbb{G}$, $V_k$ is a linear subspace of $\displaystyle\mathfrak{g}$, for every $k\in \{1,...,r\}$, and
\begin{itemize}
\item[] $\textrm{dim}V_k=n_k$, for every $k\in \{1,...,r\}$;
\item[] $[V_1,V_k]=V_{k+1}$, for $1\leq k\leq r-1$, and $[V_1,V_r]=\{0\}$.
\end{itemize}
In this setting the symbol $[V_1,V_k]$ denotes the subspace of $\mathfrak{g}$ generated by the commutators $[X,Y]$, where $X\in V_1$ and $Y\in V_k$.\par
\smallskip
\textbf{The notion of Carnot group and subelliptic Laplacian on $\mathbb{G}$}. A\textit{Carnot group} is a homogeneous group $\mathbb{G}$ such that the Lie algebra $\mathfrak{g}$ associated to $\mathbb{G}$ is stratified.\par
 Moreover, the \textit{subelliptic Laplacian} operator on $\mathbb{G}$ is the second-order differential operator, given by
$$
\Delta_{\mathbb{G}}:=\displaystyle\sum_{k=1}^{n_1}X_k^2,
$$
where $\{X_1,...,X_{n_1}\}$ is a basis of $V_1$. We shall denote by $$\nabla_\mathbb{G}:=(X_1,...,X_{n_1})$$ the
related \textit{horizontal gradient}.\par
\smallskip
\indent \textbf{Critical Sobolev inequality}. A crucial role in the functional analysis on Carnot groups is played by the
following Sobolev-type inequality
\begin{equation}\label{folland}
\int_{D}|u(\xi)|^{2^*}\,d\xi\leq C\int_{D}|\nabla_{\mathbb{G}} u(\xi)|^2\,d\xi,\,\quad\forall\, u\in C^{\infty}_0(D)
\end{equation}
due to Folland (see \cite{Fo}). In the above expression $C$ is a positive constant (independent of $u$) and
$$
2^{*}:=\frac{2\textrm{dim}_h{\mathbb{G}}}{\textrm{dim}_h{\mathbb{G}}-2},
$$
is the \textit{critical Sobolev exponent}. Inequality (\ref{folland}) ensures that if $D$ is a bounded open (smooth) subset of $\mathbb{G}$, then the function
\begin{equation}\label{norm}
u\mapsto \|u\|_{S^1_0(D)}:=\left(\int_{D}|\nabla_{\mathbb{G}} u(\xi)|^2\,d\xi\right)^{1/2}
\end{equation}
is a norm in $C^{\infty}_0({D})$.\par
\smallskip
\textbf{Folland-Stein space}. We shall denote by $S^1_0(D)$ the \textit{Folland-Stein space} defined
as the completion of $C^{\infty}_0({D})$ with respect to the norm $\|\cdot\|_{S^1_0(D)}$. The exponent $2^{*}$ is critical for $\Delta_{\mathbb{G}}$ since, as in the classical Laplacian setting, the embedding
$S^1_0(D)\hookrightarrow L^q(D)$ is compact when $1\leq q<2^{*}$, while it is only continuous if $q=2^{*}$, see Folland and Stein \cite{FoSte} and the survey paper \cite{Lanco} for related facts.\par
\smallskip
\indent \textbf{The Heisenberg group}. The simplest example of Carnot
group is provided by the \textit{Heisenberg group}
$\mathbb{H}^n:=(\erre^{2n+1},\circ)$, where, for every
$$
p:=(p_1,...,p_{2n},p_{2n+1})\,\,\,\mbox{and}\,\,\, q:=(q_1,...,q_{2n},q_{2n+1})\in \mathbb{H}^n,
$$
the usual group operation $\circ:\mathbb{H}^n\times \mathbb{H}^n\rightarrow \mathbb{H}^n$ is given by
$$
p\circ q:=\left(p_1+q_1,...,p_{2n}+q_{2n},p_{2n+1}+q_{2n+1}+\frac{1}{2}\displaystyle\sum_{k=1}^{2n}(p_kq_{k+n}-p_{k+n}q_k)\right)
$$
and the family of dilatations has the following form
$$
\delta_\eta(p):=(\eta p_1,...,\eta p_{2n},\eta^2p_{2n+1}),\quad \forall\, \eta>0.
$$
Thus $\mathbb{H}^n$ is a $(2n+1)$-dimensional group and by (\ref{dim}) it follows that
$$
\textrm{dim}_h{\mathbb{H}^n}=2n+2,
$$
and
$$
2_h^{*}:=\displaystyle 2\left(\frac{n+1}{n}\right).
$$
\indent The Lie algebra of left invariant vector fields on $\mathbb{H}^{n}$ is denoted by $\mathfrak{h}$ and its standard basis
is given by \par
$$
X_k:=\partial_k-\frac{p_{n+k}}{2}\partial_{2n+1},\quad k\in \{1,...,n\}
$$
$$
Y_k:=\partial_{n+k}-\frac{p_{k}}{2}\partial_{2n+1},\quad k\in \{1,...,n\}
$$
$$
T:=\partial_{2n+1}.
$$

\noindent In such a case, the only non-trivial commutators relations are
$$
[X_k,Y_k]=T,\quad \forall\, k\in \{1,...,n\}.
$$
Finally, the stratification of $\mathfrak{h}$ is given by
$$
\mathfrak{h}=\textrm{span}\{X_1,...,X_n,Y_1,...,Y_n\}\oplus \textrm{span}\{T\}.
$$
We denote by $\mathbb{H}^1_0(D)$ the Folland-Stein space in the Heisenberg group setting, as well as by $\Delta_{\mathbb{H}^n}$ the {Kohn-Laplacian} operator on $\mathbb{H}^n$.\par
\indent We cite the monograph \cite{BLU} for a nice introduction to Carnot groups and \cite{MRS} for related topics on variational methods used in this paper.
\section{The Main Result and some preliminary Lemmas}\label{Section3}
The aim of this section is to prove that, under natural assumptions on the nonlinear term $f$, weak solutions to problem $(P_{\lambda}^{f})$ below do exist.
More precisely, the main result is an existence theorem for
equations
driven by the subelliptic Laplacian, as stated here below.
\begin{theorem}\label{MoReGeneral}
Let $D$ be a smooth and bounded domain of the {Carnot group}
$\mathbb{G}$ of homogeneous dimension ${\rm dim}_h{\mathbb{G}}\geq 3$ and let
$f:D\times\erre\rightarrow\erre$ be a Carath\'{e}odory function such that
\begin{equation}\label{MoReGeneralg}
|f( \xi, t)|\leq \alpha(\xi)+\beta(\xi)|t|^{p}\,\, \mbox{almost everywhere in}\,\, D\times\erre,
\end{equation}
where
$$\alpha\in L^{\frac{\gamma 2^{*}}{\gamma 2^{*}-1}}(D)\qquad\mbox{and}\qquad \beta\in L^{\frac{1}{1-\gamma}}(D)$$
with $\gamma\in (2/2^*,1)$, $p\in (1,\gamma 2^*-1)$, and $2^*:=\displaystyle\frac{2{\rm dim}_h{\mathbb{G}}}{{\rm dim}_h{\mathbb{G}}-3}$. Furthermore, let
\begin{equation}\label{MoReGeneral3}
0<\lambda<\frac{(p-1)^{\frac{p-1}{p}}}{p\kappa_{1,\gamma}^{\frac{p-1}{p}}\kappa_{2,\gamma}^{\frac{p+1}{p}}\|\alpha\|^{\frac{p-1}{p}}_{L^{\frac{\gamma 2^{*}}{\gamma 2^{*}-1}}(D)}\|\beta\|_{L^{\frac{1}{1-\gamma}}(D)}^{\frac{1}{p}}},
\end{equation}
 where $\kappa_{1,\gamma}$ and $\kappa_{2,\gamma}$ denote the embedding constants of the Folland-Stein space $S^1_0(D)$ in $L^{{\gamma 2^{*}}}(D)$ and $L^{\frac{p+1}{\gamma}}(D)$, respectively. Then
the following subelliptic parametric problem
$$
(P_{\lambda}^f)\,\,\,\,\,\,\,\,\,\,\left\{
\begin{array}{ll}
-\Delta_{\mathbb{G}} u=\displaystyle \lambda f(\xi,u) &  \mbox{\rm in } D\\
u|_{\partial D}=0, &
\end{array}\right.
$$
has a weak solution $u_{0,\lambda}\in S^1_0(D)$ and $$\|u_{0,\lambda}\|_{S^1_0(D)}<\left(\lambda p\kappa_{2,\gamma}^{p+1}\|\beta\|_{L^{\frac{1}{1-\gamma}}(D)}\right)^{\frac{1}{1-p}}.$$
\end{theorem}
We recall that a \textit{weak solution} for the problem $(P_{\lambda}^{f})$, is a function $u:D\to \RR$ such that
$$
\mbox{
$\left\{\begin{array}{lll}
$$\displaystyle\int_{D} \langle\nabla_\mathbb{G}u(\xi),\nabla_\mathbb{G}\varphi(\xi)\rangle\,d\xi\\
\qquad \qquad \qquad\qquad = \displaystyle\lambda\displaystyle\int_D f(\xi,u(\xi))\varphi(\xi)d\xi, \,\,\,\,\,\,\,\forall\,\varphi \in S^1_0(D)$$\\
$$u\in S^1_0(D)$$.
\end{array}
\right.$}
$$

Let us consider the functional $\mathcal{J}_{\lambda}:S^1_0(D)\to \RR$ defined by
\begin{equation}\label{Funzionale}
\mathcal{J}_{\lambda}(u):=\frac{1}{2}\|u\|_{S^1_0(D)}^2-\lambda\displaystyle\int_D F(\xi,u(\xi))d\xi,\quad \forall\, u\in S^1_0(D)\,
\end{equation}
where $\lambda\in \erre$ and, as usual, we set $\displaystyle F(\xi,t):=\int_0^{t}f(\xi,\tau)d\tau$.\par
Note that, under our growth condition on $f$, the functional $\mathcal{J}_{\lambda}\in C^1(S^1_0(D))$ and its derivative at $u\in S^1_0(D)$ is given by
$$
\langle \mathcal{J}'_{\lambda}(u), \varphi\rangle = \displaystyle\int_{D} \langle\nabla_\mathbb{G}u(\xi),\nabla_\mathbb{G}\varphi(\xi)\rangle\,d\xi-\lambda\displaystyle\int_D f(\xi,u(\xi))\varphi(\xi)d\xi,
$$
for every $\varphi \in S^1_0(D)$.\par
 Thus the weak solutions of problem $(P_{\lambda}^{f})$ are exactly the critical points of the energy functional $\mathcal{J}_{\lambda}$.

Fix $\lambda>0$ and denote
$$
\Phi(u):=\|u\|_{S^1_0(D)}
\quad\mbox{and}\quad\Psi_\lambda(u):=\lambda\int_{D}F(\xi,u(\xi))d\xi,
$$
for every $u\in S^1_0(D)$.\par
Note that, thanks to condition \eqref{MoReGeneralg}, the operator $\Psi$ is well defined and sequentially weakly (upper) continuous. So the operator $\mathcal{J}_{\lambda}$ is sequentially weakly lower semicontinuous on $S^1_0(D)$.
With the above notations we can prove the next two lemmas that will be crucial in the sequel.

\begin{lemma}\label{lemmino1}
Let $\lambda>0$ and suppose that
\begin{equation}\label{Ga1}
\limsup_{\varepsilon\rightarrow 0^+}\frac{\displaystyle \sup_{v\in\Phi^{-1}([0,\varrho_0])}\Psi_\lambda(v)-\sup_{v\in \Phi^{-1}([0,\varrho_0-\varepsilon])}\Psi_\lambda(v)}{\varepsilon}<\varrho_0,
\end{equation}
for some $\varrho_0>0$. Then
\begin{equation}\label{Ga2}
\inf_{\sigma<\varrho_0}\frac{\displaystyle \sup_{v\in\Phi^{-1}([0,\varrho_0])}\Psi_{\lambda}(v)-\sup_{v\in  \Phi^{-1}([0,\sigma])}\Psi_\lambda(v)}{\varrho_0^2-\sigma^2}<\frac{1}{2}.
\end{equation}
\end{lemma}

\begin{proof}
First, by condition \eqref{Ga1} one has
\begin{equation}\label{Ga1GG}
\limsup_{\varepsilon\rightarrow 0^+}\frac{\displaystyle \sup_{v\in\Phi^{-1}([0,\varrho_0])}\Psi_\lambda(v)-\sup_{v\in \Phi^{-1}([0,\varrho_0-\varepsilon])}\Psi_\lambda(v)}{\varrho^2_0-(\varrho_0-\varepsilon)^2}<\frac{1}{2}.
\end{equation}
Indeed, if $\varepsilon \in (0,\varrho_0)$, one has \par
\smallskip
$\displaystyle
 \frac{\displaystyle \sup_{v\in\Phi^{-1}([0,\varrho_0])}\Psi_\lambda(v)-\sup_{v\in \Phi^{-1}([0,\varrho_0-\varepsilon])}\Psi_\lambda(v)}{\varrho^2_0-(\varrho_0-\varepsilon)^2}
 $
 $$
 =\frac{\displaystyle \sup_{v\in\Phi^{-1}([0,\varrho_0])}\Psi_\lambda(v)-\sup_{v\in \Phi^{-1}([0,\varrho_0-\varepsilon])}\Psi_\lambda(v)}{\varepsilon}\times
 $$
 $$
 \times\frac{-\varepsilon /\varrho_0}{\varrho_0\left[\left(1-\displaystyle\frac{\varepsilon}{\varrho_0}\right)^2-1\right]},
 $$
and
$$
\lim_{\varepsilon\rightarrow 0^+}\frac{-\varepsilon /\varrho_0}{\varrho_0\left[\left(1-\displaystyle\frac{\varepsilon}{\varrho_0}\right)^2-1\right]}=\frac{1}{2\varrho_0}.
$$
Now, by \eqref{Ga1GG} there exists $\bar\varepsilon>0$ such that
$$
\frac{\displaystyle \sup_{v\in\Phi^{-1}([0,\varrho_0])}\Psi_\lambda(v)-\sup_{v\in \Phi^{-1}([0,\varrho_0-\varepsilon])}\Psi_\lambda(v)}{\varrho^2_0-(\varrho_0-\varepsilon)^2}<\frac{1}{2},
$$
for every $\varepsilon\in ]0,\bar\varepsilon[$. Setting $\sigma_0:=\varrho_0-\varepsilon_0$ (with $\varepsilon_0\in ]0,\bar\varepsilon[$), it follows that
$$
\frac{\displaystyle \sup_{v\in\Phi^{-1}([0,\varrho_0])}\Psi_{\lambda}(v)-\sup_{v\in  \Phi^{-1}([0,\sigma_0])}\Psi_\lambda(v)}{\varrho_0^2-\sigma^2_0}<\frac{1}{2},
$$
and thus inequality \eqref{Ga2} is verified.
\end{proof}

\begin{lemma}\label{lemmino2}
Let $\lambda>0$ and suppose that condition \eqref{Ga2} holds. Then
\begin{equation}\label{VGa1}
\inf_{u\in\Phi^{-1}([0,\varrho_0))}
\frac{\displaystyle\sup_{v\in\Phi^{-1}([0,\varrho_0])}\Psi_\lambda(v)-\Psi_\lambda(u)}{\varrho_0^2- \|u\|_{S^1_0(D)}^2}<\frac{1}{2}.
\end{equation}
\end{lemma}
\begin{proof}
Assumption \eqref{Ga2} yields
\begin{equation}\label{Ga123}
\displaystyle \sup_{v\in\Phi^{-1}([0,\sigma_0])}\Psi_{\lambda}(v)>\displaystyle \sup_{v\in\Phi^{-1}([0,\varrho_0])}\Psi_{\lambda}(v)-\frac{1}{2}(\varrho_0^2-\sigma^2_0),
\end{equation}
for some $0<\sigma_0<\varrho_0$. Thanks to the weakly regularity of the functional $\Psi_\lambda$, since
\begin{equation*}\label{Ga1234}
\displaystyle \sup_{v\in\Phi^{-1}([0,\sigma_0])}\Psi_{\lambda}(v)=\displaystyle \sup_{\|v\|_{S^1_0(D)}=\sigma_0}\Psi_{\lambda}(v),
\end{equation*}
by \eqref{Ga123} there exists $u_0\in S^1_0(D)$ with $\|u_0\|_{S^1_0(D)}=\sigma_0$ such that
\begin{equation}\label{Ga1235}
\displaystyle \Psi_{\lambda}(u_0)>\displaystyle \sup_{v\in\Phi^{-1}([0,\varrho_0])}\Psi_{\lambda}(v)-\frac{1}{2}(\varrho_0^2-\sigma^2_0),
\end{equation}
that is,
\begin{equation}\label{VGa1Finale}
\frac{\displaystyle\sup_{v\in\Phi^{-1}([0,\varrho_0])}\Psi_\lambda(v)-\Psi_\lambda(u_0)}{\varrho_0^2- \|u_0\|_{S^1_0(D)}^2}<\frac{1}{2},
\end{equation}
with $\|u_0\|_{S^1_0(D)}=\sigma_0$. The proof is now complete.
\end{proof}

\section{Proof of Theorem \ref{MoReGeneral}}
For the proof of our result, before, we note that
problem~$(P_{\lambda}^{f})$ has a variational structure. Indeed, it is the Euler-Lagrange equation of the functional $\mathcal{J}_{\lambda}$.\par
 Hence, fix
 \begin{equation}\label{range}
 \lambda\in \left(0,\frac{(p-1)^{\frac{p-1}{p}}}{p\kappa_{1,\gamma}^{\frac{p-1}{p}}\kappa_{2,\gamma}^{\frac{p+1}{p}}\|\alpha\|^{\frac{p-1}{p}}_{L^{\frac{\gamma 2^{*}}{\gamma 2^{*}-1}}(D)}\|\beta\|_{L^{\frac{1}{1-\gamma}}(D)}^{\frac{1}{p}}}\right),
 \end{equation}
 and let us consider $0<\varepsilon<\varrho$.
 Setting
 $$
 \Lambda_\lambda(\varepsilon,\varrho):=\displaystyle\frac{\displaystyle \sup_{v\in\Phi^{-1}([0,\varrho])}\Psi_\lambda(v)-\sup_{v\in \Phi^{-1}([0,\varrho-\varepsilon])}\Psi_\lambda(v)}{\varepsilon},
 $$
 one has
 $$
 \Lambda_\lambda(\varepsilon,\varrho)\leq\frac{1}{\varepsilon}\left|\displaystyle \sup_{v\in\Phi^{-1}([0,\varrho])}\Psi_\lambda(v)-\sup_{v\in \Phi^{-1}([0,\varrho-\varepsilon])}\Psi_\lambda(v)\right|.
 $$
Moreover, it follows that
 $$
 \Lambda_\lambda(\varepsilon,\varrho)\leq\sup_{v\in \Phi^{-1}([0,1])}\int_D\left|\int_{(\varrho-\varepsilon)v(\xi)}^{\varrho v(\xi)}\lambda\frac{|f(\xi,t)|}{\varepsilon}dt\right|d\xi.
 $$
 Now the growth condition \eqref{MoReGeneralg} yields
 $$
\sup_{v\in \Phi^{-1}([0,1])}\int_D\left|\int_{(\varrho-\varepsilon)v(\xi)}^{\varrho v(\xi)}\lambda\frac{|f(\xi,t)|}{\varepsilon}dt\right|d\xi \leq \sup_{v\in \Phi^{-1}([0,1])}\int_D \lambda\alpha(\xi)|v(\xi)|d\xi
$$
$$
+\sup_{v\in \Phi^{-1}([0,1])}\int_D \frac{\lambda\beta(\xi)}{p+1}\left(\frac{\varrho^{p+1}-(\varrho-\varepsilon)^{p+1}}{\varepsilon}\right)|v(\xi)|^{p+1}d\xi.
$$
 \noindent Since the Folland-Stein space $S^1_0(D)$ is compactly embedded in $L^{q}(D)$, for every $q\in [1,2^*)$,
 bearing in mind that
$$\lambda \alpha\in L^{\frac{\gamma 2^{*}}{\gamma 2^{*}-1}}(D)\qquad\mbox{and}\qquad \lambda\beta\in L^{\frac{1}{1-\gamma}}(D),$$
 the above inequality yields
 $$
 \Lambda_\lambda(\varepsilon,\varrho)\leq \kappa_{1,\gamma}\|\lambda\alpha\|_{L^{\frac{\gamma 2^{*}}{\gamma 2^{*}-1}}(D)}+\frac{\kappa_{2,\gamma}^{p+1}}{p+1}\|\lambda\beta\|_{L^{\frac{1}{1-\gamma}}(D)}\left(\frac{\varrho^{p+1}-(\varrho-\varepsilon)^{p+1}}{\varepsilon}\right).
 $$
 \indent Thus passing to the limsup, as $\varepsilon\rightarrow 0^+$, we get
 \begin{equation}\label{VGa1Glu1}
 \limsup_{\varepsilon\rightarrow 0^+}\Lambda_\lambda(\varepsilon,\varrho)<\kappa_{1,\gamma}\|\lambda\alpha\|_{L^{\frac{\gamma 2^{*}}{\gamma 2^{*}-1}}(D)}+\kappa_{2,\gamma}^{p+1}\|\lambda\beta\|_{L^{\frac{1}{1-\gamma}}(D)}\varrho^{p}.
 \end{equation}
 Now, consider the real function
 $$
 \varphi_\lambda(\varrho):=\kappa_{1,\gamma}\|\lambda\alpha\|_{L^{\frac{\gamma 2^{*}}{\gamma 2^{*}-1}}(D)}+\kappa_{2,\gamma}^{p+1}\|\lambda\beta\|_{L^{\frac{1}{1-\gamma}}(D)}\varrho^{p}-\varrho,
 $$
 for every $\varrho>0$.\par
   \noindent It is easy to see that $\inf_{\varrho>0}\varphi_\lambda(\varrho)$ is attained at
 $$
 \varrho_{0,\lambda}:=\left(\lambda p\kappa_{2,\gamma}^{p+1}\|\beta\|_{L^{\frac{1}{1-\gamma}}(D)}\right)^{\frac{1}{1-p}}.
 $$
 and, by \eqref{range}, one has
 $$
 \inf_{\varrho>0}\varphi_\lambda(\varrho)<0.
 $$
 \indent Hence inequality \eqref{VGa1Glu1} yields
 $$
 \limsup_{\varepsilon\rightarrow 0^+}\Lambda_\lambda(\varepsilon,\varrho)<\varrho_{0,\lambda}.
 $$
 \noindent Now, it follows by Lemmas \ref{lemmino1} and \ref{lemmino2} that
 \begin{equation*}\label{VGa1Glu}
\inf_{u\in\Phi^{-1}([0,\varrho_{0,\lambda}))}
\frac{\displaystyle\sup_{v\in\Phi^{-1}([0,\varrho_{0,\lambda}])}\Psi_\lambda(v)-\Psi_\lambda(u)}{\varrho_{0,\lambda}^2- \|u\|_{S^1_0(D)}^2}<\frac{1}{2}.
\end{equation*}
The above relation implies that there exists $w_{\lambda}\in S^1_0(D)$ such that
$$
\Psi_\lambda(u)\leq \displaystyle\sup_{v\in\Phi^{-1}([0,\varrho_{0,\lambda}])}\Psi_\lambda(v)<\Psi_\lambda(w_\lambda)+\frac{1}{2}(\varrho_{0,\lambda}^2-\|w_\lambda\|^2_{S^1_0(D)}),
$$
for every $u\in\Phi^{-1}([0,\varrho_{0,\lambda}])$.\par
 Thus
 \begin{equation}\label{John}
 \mathcal{J}_{\lambda}(w_{\lambda}):=\frac{1}{2}\|w_\lambda\|_{S^1_0(D)}^2-\Psi_\lambda(w_\lambda)<\frac{\varrho^2_{0,\lambda}}{2}-\Psi_\lambda(u),
 \end{equation}
 for every $u\in\Phi^{-1}([0,\varrho_{0,\lambda}])$.\par
 \indent Since the energy functional $\mathcal{J}_{\lambda}$ is sequentially weakly lower semicontinuous, its restriction on $\Phi^{-1}([0,\varrho_{0,\lambda}])$ has a global
 minimum $u_{0,\lambda}\in \Phi^{-1}([0,\varrho_{0,\lambda}])$.\par
  \indent Note that $u_{0,\lambda}$ belongs to $\Phi^{-1}([0,\varrho_{0,\lambda}))$. Indeed, if
 $\|u_{0,\lambda}\|_{S^1_0(D)}=\varrho_{0,\lambda}$, by \eqref{John}, one has
  \begin{equation*}
 \mathcal{J}_{\lambda}(u_{0,\lambda})=\frac{\varrho^2_{0,\lambda}}{2}-\Psi_\lambda(u_{0,\lambda})>\mathcal{J}_{\lambda}(w_{\lambda}),
 \end{equation*}
 \noindent which is a contradiction.\par
 \indent In conclusion, it follows that $u_{0,\lambda}\in S^1_0(D)$ is a local minimum for the energy functional $\mathcal{J}_{\lambda}$ with $$\|u_{0,\lambda}\|_{S^1_0(D)}<\varrho_{0,\lambda},$$ hence in particular, a weak solution of problem $(P_{\lambda}^{f})$.
 This completes the proof.

\begin{remark}\rm{
A crucial step in our approach is the explicit computation of the embedding
constants $\kappa_{i,\gamma}$ that naturally appear in Theorem \ref{MoReGeneral} and its
consequences. In the special case of the Heisenberg
group $\mathbb{H}^n$ an explicit expression of these quantities can be obtained by using
 the best constant in the Sobolev inequality
 \begin{equation}\label{folland2}
\int_{D}|u(\xi)|^{2^*_h}\,d\xi\leq C\int_{D}|\nabla_{\mathbb{H}^n} u(\xi)|^2\,d\xi,\,\quad\forall\, u\in C^{\infty}_0(D)
\end{equation}
 that was determined by Jerison and Lee in \cite[Corollary C]{JLe}.}
\end{remark}

\begin{remark}\rm{It is clear that Theorem \ref{FerraraMolicaBisci22Generale2} is a simple consequence of Theorem \ref{MoReGeneral}. Indeed, preserving our notations and assuming that
\begin{equation*}
0<\kappa<\frac{(p-1)^{\frac{p-1}{p}}}{pc_{1,\gamma}^{\frac{p-1}{p}}c_{2,\gamma}^{\frac{p+1}{p}}}|D|^{\frac{1-\gamma}{p}+
\frac{(1-p)(\gamma 2^{*}_h-1)}{p\gamma 2^{*}_h}
},
\end{equation*}
it is easy to note that
$$
\frac{(p-1)^{\frac{p-1}{p}}}{pc_{1,\gamma}^{\frac{p-1}{p}}c_{2,\gamma}^{\frac{p+1}{p}}\|\kappa\|^{\frac{p-1}{p}}_{L^{\frac{\gamma 2^{*}}{\gamma 2^{*}-1}}(D)}\|\kappa\|_{L^{\frac{1}{1-\gamma}}(D)}^{\frac{1}{p}}}>1.
$$
Since all the assumptions of Theorem \ref{MoReGeneral} have been verified (with $\lambda=1$) the conclusion of Theorem \ref{FerraraMolicaBisci22Generale2} immediately follows.
}
\end{remark}

A special case of Theorem \ref{MoReGeneral} reads as follows.

\begin{corollary}\label{MoReGeneralCor}
Let $D$ be a smooth and bounded domain of the {Carnot group}
$\mathbb{G}$ of homogeneous dimension ${\rm dim}_h{\mathbb{G}}\geq 3$ and let
$f:D\times\erre\rightarrow\erre$ be a Carath\'{e}odory function such that condition \eqref{MoReGeneralg} holds.
Assume that
\begin{equation}\label{MoReGeneral3D}
\|\alpha\|^{p-1}_{L^{\frac{\gamma 2^{*}}{\gamma 2^{*}-1}}(D)}\|\beta\|_{L^{\frac{1}{1-\gamma}}(D)}<\frac{(p-1)^{p-1}}{p\kappa_{1,\gamma}^{{p-1}}\kappa_{2,\gamma}^{{p+1}}}.
\end{equation}
 Then
the following subelliptic problem
$$
(P_{f})\,\,\,\,\,\,\,\,\,\,\left\{
\begin{array}{ll}
-\Delta_{\mathbb{G}} u=\displaystyle f(\xi,u) &  \mbox{\rm in } D\\
u|_{\partial D}=0, &
\end{array}\right.
$$
has a weak solution $u_{0}\in S^1_0(D)$ and $$\|u_{0}\|_{S^1_0(D)}<\left( p\kappa_{2,\gamma}^{p+1}\|\beta\|_{L^{\frac{1}{1-\gamma}}(D)}\right)^{\frac{1}{1-p}}.$$
\end{corollary}

\begin{remark}\rm{A special case of Corollary \ref{MoReGeneralCor} in the Euclidean setting has been proved in \cite{AC} by exploiting the variational principle obtained by Ricceri in \cite{R2}.
}
\end{remark}

\indent In conclusion, we present a direct application of our main result.

\begin{example}\label{esempiuccio}\rm{
Let $D$ be a smooth and bounded domain of a Carnot group $\mathbb{G}$ with ${\rm dim}_h{\mathbb{G}}\geq 3$ and let $$\alpha\in L^{\frac{\gamma 2^{*}}{\gamma 2^{*}-1}}(D)\setminus\{0\},$$
with $\gamma\in (2/2^*,1)$.\par
By virtue of Theorem \ref{MoReGeneral}, there exists an open interval $\Lambda\subset (0,+\infty)$ such that for every $\lambda\in \Lambda$,
the following problem
$$
\left\{
\begin{array}{ll}
-\Delta_{\mathbb{G}} u=\lambda\displaystyle (\alpha(\xi)+|u|^p) &  \mbox{\rm in } D\\
u|_{\partial D}=0, &
\end{array}\right.
$$
where $p\in (1,\gamma 2^*-1)$,
admits at least one non-trivial weak solution $u_{0,\lambda}\in S^1_0(D)$ such that
$$\|u_{0,\lambda}\|_{S^1_0(D)}<\left(\lambda p\kappa_{2,\gamma}^{p+1}|D|^{1-\gamma}\right)^{\frac{1}{1-p}}.$$
More precisely, a concrete expression of the interval $\Lambda$ is given by
$$
\Lambda:=\left(0,\frac{(p-1)^{\frac{p-1}{p}}|D|^{\frac{\gamma-1}{p}}}{p\kappa_{1,\gamma}^{\frac{p-1}{p}}\kappa_{2,\gamma}^{\frac{p+1}{p}}\|\alpha\|^{\frac{p-1}{p}}_{L^{\frac{\gamma 2^{*}}{\gamma 2^{*}-1}}(D)}}\right).
$$
}
\end{example}

\indent {\bf Acknowledgements.}  The manuscript was realized within the auspices of the INdAM - GNAMPA Project 2015 {\it Modelli ed equazioni non-locali di tipo frazionario} and the SRA grants P1-0292, J1-7025, J1-6721, and J1-5435.


\begin{thebibliography}{99}

\bibitem{AC}{\sc G. Anello and G. Cordaro}, \emph{An existence and localization theorem for the solutions of a
Dirichlet problem}, {Ann. Pol. Math.} \textbf{83} (2004), 107-112.

\bibitem{BK}{\sc M.Z. Balogh and A. Krist\'{a}ly}, \emph{Lions-type compactness and Rubik actions on the Heisenberg group}, {Calc. Var. Partial Differential Equations} \textbf{48} (2013), 89-109.

\bibitem{BLU}{\sc A. Bonfiglioli, E. Lanconelli, and F. Uguzzoni},
Stratified Lie groups and potential theory for their sub-Laplacians,
Springer Monographs in Mathematics. \textit{Springer}, Berlin, 2007.

\bibitem{BFP}
{\sc S. Bordoni, R. Filippucci, and P. Pucci}, \emph{Nonlinear elliptic inequalities with gradient terms on the
Heisenberg group},
{{Nonlinear Analysis}}
\textbf{121} (2015), 262-279.

\bibitem{DM}{\sc L. D'Ambrosio and E. Mitidieri}, \emph{Entire solutions of quasilinear elliptic systems on Carnot groups}, Proceedings of the Steklov Institute of Mathematics \textbf{283} (2013), 3-19.

\bibitem{Fo}{\sc G.B. Folland}, \emph{Subelliptic estimates and function spaces on nilpotent Lie groups}, Ark.
Mat. \textbf{13} (1975), 161-207.

\bibitem{FoSte}{\sc G.B. Folland and E.M. Stein}, \emph{Estimates for the $\bar\partial_b$ complex and analysis on the Heisenberg group}, Commun.
Pure Appl. Math. \textbf{27} (1974), 429-522.

\bibitem{GaLa}{\sc N. Garofalo and E. Lanconelli}, \emph{Existence and nonexistence results for semilinear equations on the Heisenberg group}, Indiana Univ. Math. J. \textbf{41} (1992), 71-98.

\bibitem{JLe}{\sc D. Jerison and J.M. Lee}, \emph{Extremals of the Sobolev inequality on the Heisenberg group and the CR Yamabe
problem}, J. Differential Geom. \textbf{29} (1989), 303-343.

\bibitem{Lanco}
{\sc E. Lanconelli}, \emph{Nonlinear equations on Carnot groups and curvature problems for CR manifolds},
{{Rend. Mat. Acc. Lincei}} {\textbf{14}} (2003), {227-238}.

\bibitem{Lo1}
{\sc A. Loiudice}, \emph{Semilinear subelliptic problems with critical growth on Carnot groups},
{{Manuscripta Math.}} {\textbf{124}} (2007), {247-259}.

\bibitem{Mingio0}
{\sc J.J. Manfredi and G. Mingione}, \emph{Regularity results for quasilinear elliptic equations in the Heisenberg group},
{{Math. Ann.}} {\textbf{339}} (2007), {485-544}.

\bibitem{Mingio}
{\sc G. Mingione, A. Zatorska-Goldestein, and X. Zhong}, \emph{Gradient regularity for elliptic equations in the Heisenberg group},
{{Adv. in Math.}} {\textbf{222}} (2009), {62-129}.

\bibitem{MR}
{\sc G. Molica Bisci and M. Ferrara}, \emph{Subelliptic and parametric equations on Carnot groups},
{{Proc. Amer. Math. Soc.}}
\textbf{144} (2016), no. 7, 3035--3045.

\bibitem{MRe}
{\sc G. Molica Bisci and D. Repov\v{s}}, \emph{Yamabe-type equations on Carnot groups},
{{Potential Anal.}}
\textbf{46} (2017), no. 2, 369--383. \href{https://arxiv.org/abs/1705.10100}{arXiv:1705.10100 [math.AP]}

\bibitem{MRS}
{\sc G. Molica Bisci, V. R\u{a}dulescu, and R. Servadei}, {Variational Methods for Nonlocal Fractional Problems},
Encyclopedia of Mathematics and its Applications, No.~\textbf{142}, \emph{Cambridge University Press}, Cambridge, 2016.

\bibitem{PiVa}{\sc A. Pinamonti and E. Valdinoci}, \emph{A Lewy-Stampacchia estimate for variational
inequalities in the Heisenberg group},
{Rend. Istit. Mat. Univ. Trieste} \textbf{45} (2013), 1-22.

\bibitem{R2}{\sc B. Ricceri}, \emph{A general variational principle and some of its applications}, {J. Comput. Appl. Math.}, Special Issue on Fixed point theory with
applications in Nonlinear Analysis \textbf{113} (2000), 401-410.

\end{thebibliography}
\end{document}